\documentclass[a4paper,12pt]{article}
\usepackage{latexsym}
\usepackage{amsmath}
\usepackage{amssymb}
\usepackage{enumerate}
\usepackage{mathrsfs}

\makeatletter
 \long\def\@makefntext#1{\parindent 1em\noindent
  \hbox to 1.1em{\hss $^{\@thefnmark}$}#1}
\makeatother

\usepackage{theorem}
\newtheorem{theorem}{Theorem}[section]
\newtheorem{proposition}[theorem]{Proposition}
\newtheorem{lemma}[theorem]{Lemma}

\newtheorem{claim}[theorem]{Claim}

\theorembodyfont{\rmfamily}
\newtheorem{proof}{\textmd{\textit{Proof.}}}

\newtheorem{remark}[theorem]{Remark}

\newtheorem{acknowledgement}{\textmd{\textit{Acknowledgements.}}}

\makeatletter

\@addtoreset{equation}{section}
\makeatother

\newcommand{\qedd}{\hfill \Box}

\newcommand{\R}{\ensuremath{\mathbb{R}}}
\newcommand{\N}{\ensuremath{\mathbb{N}}}

\def\supp{\mathop{\mathrm{supp}}\nolimits}

\newcommand{\scal}{({\rm C}) }
\newcommand{\lip}[1]{({\rm C}_{#1})}

\title{Isoperimetric profile of radial probability measures on Euclidean spaces}
\author{Asuka Takatsu\thanks{Graduate School of Mathematics, Nagoya University, Nagoya 464-8602,
Japan ({\sf takatsu@math.nagoya-u.ac.jp});
Supported in part by the Grant-in-Aid for Young Scientists (B) 24740042
.}}
\date{\empty}
\begin{document}
\maketitle
\begin{abstract}
We derive the isoperimetric profile of Gaussian type for an absolutely continuous probability measure 
on Euclidean spaces with respect to the Lebesgue measure, whose density is a radial function.
The key is a generalization of the Poincar\'e limit which asserts that 
the $n$-dimensional Gaussian measure is approximated by the projections of the uniform probability measure  on the Euclidean sphere of appropriate radius to the first $n$-coordinates as the dimension diverges to infinity. 
The generalization is done by replacing the projections with certain maps.
\footnote[0]{
{\bf  keywords:}Isoperimetric profile, Poincar\'e limit. } 
\end{abstract}
\section{Introduction}
The isoperimetric profile of a Borel probability measure $\mu$ on $\R^n$ describes a relation between the volume $\mu[A]$ and the {\it boundary measure} 
$\mu^+[A]:=\varliminf_{\varepsilon \downarrow 0} (\mu[A^\varepsilon]-\mu[A])/ \varepsilon$ of $A \subset \R^n$, 
where $A^\varepsilon := \left\{ x \in \R^n \ |\ \inf_{a\in A} |x-a| < \varepsilon \right\}$ 
denotes the $\varepsilon$-neighborhood of $A$ with respect to the standard Euclidean norm $|\cdot|$. 
Throughout this note, any subset of $\R^n$ is assumed to be Borel.
Precisely, the {\it isoperimetric profile}  $I[\mu]$ of $\mu$ is a function on $[0,1]$ defined  by  
\[
I[\mu](a):=\inf \left\{ \mu^+[A] \ |\ A \subset \R^n \text{ with } \mu[A]=a\right\}. 
\]

Let $A_n$ denote the boundary measure of the unit ball in $\R^n$ with respect to the Lebesgue measure.
For a measurable, nonnegative function $f$ on $(0,\infty)$ satisfying 
\[
 M_n^f:=\frac{1}{A_n}\int_{0}^\infty f(r) r^{n-1} dr <\infty, 
\]
the $n$-dimensional {\it radial probability measure $\mu_n^f$ with density $f$} is the absolutely continuous probability measure on $\R^n$ 
with density
\[
\frac{d \mu_n^f}{dx}(x)=\frac{1}{M_n^f} f(|x|) 
\]
with respect to the $n$-dimensional Lebesgue measure.
For example, the $n$-dimensional Gaussian measure $\gamma_n$ is the radial probability measure with density $g(r):=\exp(-r^2/2)$, 
and its isoperimetric profile was provided by Borell~\cite{Bo} and Sudakov--Tsirel'son~\cite{ST} independently of the form  
\[
I[\gamma_n](a)=I[\gamma_1](a)=G'\left( G^{-1}(a)\right), \quad G(r):=\int_{-\infty}^r (2\pi)^{-1/2}g(s) ds=\gamma_1[(-\infty,r]].
\]
The proof relies on the approximation procedure, so-called Poincar\'e limit: 
let $S_N$ be the $(N-1)$-dimensional Euclidean sphere of radius $N^{1/2} $ and  $v_N$ be the uniform probability measure on $S_N$.
We consider the orthogonal projection from $\R^N$ to the first $n$-coordinates, and denote by $P_{n,N}$ the restriction of it on $S_N$.
Then $\gamma_n$ is obtained as the limit of $(P_{n,N})_{\sharp}v_N$ as $N \to \infty$, 
where  $(P_{n,N})_{\sharp}v_N$ denotes the push-forward measure of $v_N$ by $P_{n,N}$, namely 
$(P_{n,N})_{\sharp}v_N[A]=v_N[P_{n,N}^{-1}(A)]$ for any $A \subset \R^n$.

The aim of this note is to derive the isoperimetric profile of Gaussian type for $\mu_n^f$, 
that is, estimate $I[\mu_n^f]$ below by $I[\gamma_1]$.
To do this,  let us generalize the Poincar\'e limit by replacing $P_{n,N}$ with $P_{n,N}^\rho:= s_n^\rho \circ P_{n,N}$, 
where $s_n^\rho$ is the map on $\R^n$ defined as 
\[
 s_n^\rho(x):=
\begin{cases}
\rho(|x|)x & \text{if } x \neq 0,\\ 
0 & \text{if } x = 0
\end{cases}
\]
for a function $\rho$ on $(0,\infty)$ satisfying the following condition.
\begin{itemize}
\item[(C)]
$\rho$ is a $C^1$, positive function on $(0, \infty)$ and $s_1^{\rho}$ is strictly increasing.  
\end{itemize}
\begin{theorem}\label{poin}
For a function $\rho$ satisfying~$\scal$, 
let $\sigma$ be the inverse function of $s_1^\rho$.
For any $x \in \R^n \setminus\{0\}$,   
$\{f^\rho_{n,N}(x):=d (P_{n,N}^\rho{}_\sharp v_N)(x) /dx\}_{N\in \N}$ converges to 
\[
f^\rho_n(x):=
\begin{cases}
\displaystyle
(2\pi)^{-n/2} \exp\left( -\frac{\sigma(|x|)^2}2\right) \left\{ \frac{\sigma(|x|)}{|x|} \right\}^{n-1} \sigma'(|x|)
 & \text{if } x \in s_n^\rho (\R^n \setminus \{0\}),\\ 
0 & \text{otherwise} 
\end{cases}
\]
as $N \to \infty$.
The function $f_n^\rho$ has unit mass on $\R^n$ with respect to the Lebesgue measure
and hence $\{P_{n,N}^\rho{}_\sharp v_N\}_{N\in \N}$ converges weakly to
the absolutely continuous probability measure $\nu_n^\rho$ on $\R^n$ such that $d\nu_n^\rho/dx=f_n^\rho$ as $N \to \infty$.  
\end{theorem}
Theorem~\ref{poin} for the case of $\rho \equiv 1$ recovers the original Poincar\'e limit.
A radial probability measure $\mu_n^f$ is said to be a {\it generalized Poincar\'e limit} 
if there exists a function $\rho$ satisfying~\scal  such that $\mu_n^f=\nu_n^\rho$.
To estimate the isoperimetric profile of $\mu_n^f=\nu_n^\rho$, 
we impose an additional condition on $\rho$.
\begin{itemize}
\item[(C$_n$)]
The map $s_n^\rho$ is Lipschitz continuous.
\end{itemize}
\begin{theorem}\label{isop}
For $m=1$ and $n$, 
let $\mu_m^f$ be the generalized Poincar\'e limit with $\rho_m$ satisfying~$\lip{m}$. 
Then it holds for any $a \in [0,1]$ with $a \neq 1/2$ that    
\[
I[\mu_n^f](a) \geq \frac{1}{L_n} I[\gamma_1](a), 
\]
where $L_n$ is the smallest Lipschitz constant of $s_n^{\rho_n}$.
Moreover, if $\varlimsup_{r \downarrow 0} f(r) \in (0,\infty)$, 
then the above inequality also holds true for $a=1/2$.
\end{theorem}
Theorem~\ref{isop} for the case of $f(r)=\exp(-r^2/2)$ with $\rho_m \equiv 1$ corresponds to the result of the Gaussian measure.

This note is organized as follows.
Section~2 concerns Theorem~\ref{poin} which is a generalization of the Poincar\'e limit.
In Section~3, we prove Theorem~\ref{isop}, namely derive the isoperimetric profile of Gaussian type for a radial probability measure.
Section~4 provides criteria and examples of $\mu_n^f$ which is applicable to Theorems~\ref{poin}, \ref{isop}.

\begin{acknowledgement}
The author would like to express her gratitude to Karl-Theodor Sturm for his hospitality during her stay in Bonn, 
where the author  worked as the JSPS-IH\'ES (EPDI) fellow, 
as well as for his invitation to Winter School on ``New Trends in Stochastic and Geometric Analysis".
She thanks all the participants in the school, especially Martin Huesmann, for their valuable discussions.
She would also like to thank Michel Ledoux, Shin-ichi Ohta and Shouhei Honda for remarkable comments.
\end{acknowledgement}
\section{Generalized Poincar\'e limit}
In this section, we always assume that $\rho$ satisfies~\scal and 
$\sigma$ is the inverse function of $s_1^\rho$.
We moreover define the map $\Sigma$ on $\R^n$ by 
\begin{align*}
\Sigma(x)=
\begin{cases}
\displaystyle
 \frac{\sigma(|x|)}{| x|} x &\text{if\ } x \in s_n^\rho(\R^n \setminus\{0\}), \\
0 &\text{otherwise.} 
\end{cases}
\end{align*}
We then have $\Sigma \circ s_n^\rho (x)=x$ for any $x  \in\R^n$ and 
\[
s_n^\rho(\R^n \setminus\{0\})=s_n^\rho(\R^n) \setminus\{0\}=\{x \in \R^n \ |\ |x| \in s_1^\rho((0,\infty))\}.
\] 
Let $V_n$ denote the volume of the unit ball in $\R^n$ with respect to the Lebesgue measure. 

For any $x \in P_{n,N}^\rho(S_N)\setminus\{0\}$, the Lebesgue differentiation theorem yields    
\begin{equation}\label{den}
f^\rho_{n,N}(x)
=\lim_{\varepsilon \downarrow 0}\frac{1}{\varepsilon^n  V_{n}}\int_{B_\varepsilon(x)} f^\rho_{n,N}(x')dx'
=\lim_{\varepsilon \downarrow 0}\frac{1}{\varepsilon^n  V_{n}} v_N[(P_{n,N}^\rho)^{-1}(B_\varepsilon(x))],
\end{equation}
where $B_\varepsilon(x)$ is the open ball in $\R^n$ with center $x$ and  radius $\varepsilon$. 
We compute the right-hand side in~\eqref{den}.
\begin{lemma}\label{express}
For any $x \in P_{n,N}^{\rho}(S_N)\setminus\{0\}$, we have  
\[
f^\rho_{n,N}(x) 
=\frac{ A_{N-n}}{N^{n/2} A_N} 
\left\{ 1-\frac{\sigma(|x|)^2}{N}\right\}^{(N-n-2)/2} 
\left\{ \frac{\sigma(|x|)}{|x|} \right\}^{n-1}\sigma'(|x|).
\]
\end{lemma}
\begin{proof}
We prove only the case of $n \geq 2$, however a similar argument works for the case of $n=1$.
By symmetry, we may assume that $x$ lies in the positive first coordinate axis. 

Let us consider the orthogonal projection $p_m$ from $\R^m$ to the last $(m-1)$ coordinates.
We define the functions $r^+_\varepsilon$ and $r^-_\varepsilon$ on $p_n(\Sigma (B_\varepsilon(x)))$ by 
\begin{gather*}
r^+_\varepsilon(y):= \sup \{|x'| \ | \ x'\in B^y_\varepsilon(x) \}, \quad
r^-_\varepsilon(y):= \inf \{|x'| \ | \ x' \in B^y_\varepsilon(x) \}, 
\end{gather*}
where $B_\varepsilon^y(x):= (p_n \circ \Sigma )^{-1}(y) \cap B_\varepsilon (x)$.
Since $p_N$ is a chart on a set containing  
\begin{gather*}
(P_{n,N}^\rho)^{-1}(B_\varepsilon(x))=\{(\Sigma(x'),\xi) \in \R^n \times \R^{N-n}\ | \ |\xi|^2=N-\sigma(|x'|)^2, x'\in B_\varepsilon(x) \} \subset S_N
\end{gather*}
for $\varepsilon>0$ small enough, we directly compute  
\begin{align*}
& v_N[(P_{n,N}^\rho)^{-1}(B_\varepsilon(x))] \\
=& \frac{1}{N^{(N-1)/2}A_N }\int_{p_N((P_{n,N}^\rho)^{-1}(B_\varepsilon(x)))} \left(\frac{N}{N-|u|^2}\right) ^{1/2} du \\
=&\frac{1}{N^{(N-1)/2}A_N } \int_{ p_n (\Sigma ( B_\varepsilon(x) )) } 
    \int_{\{\xi \in \R^{N-n}  \ | \ |\xi|^2=N-\sigma(|x'|)^2, \ x'\in B_\varepsilon^y(x) \}} \left( \frac{N}{N-|y|^2-|\xi|^2}\right)^{1/2} d\xi dy\\
=&\frac{A_{N-n} }{N^{(N-1)/2}A_N } \int_{ p_n (\Sigma ( B_\varepsilon(x) ) } \int_{ \{N-\sigma(r^+_\varepsilon(y))^2\}^{1/2}}^{\{ N-\sigma(r^-_\varepsilon(y))^2 \}^{1/2}} 
\left( \frac{N}{N-|y|^2-s^2} \right)^{1/2} s^{N-n-1}ds dy \\
=& \frac{A_{N-n} }{N^{(N-1)/2}A_N }  \int_{ U_\varepsilon (x)} \left(  \frac{N}{N-|y|^2-s^2} \right)^{1/2} s^{N-n-1} dy ds, 
\end{align*}
where we set 
\[
U_\varepsilon (x):=\left\{\left(p_n(\Sigma(x')),s\right) \in \R^{n-1} \times \R_{\geq 0} \ |\ s^2=N-\sigma(|x'|)^2,\ x'\in B_\varepsilon(x)  \right\}. 
\]
According to the assumption that $x$ lies in the first axis, $U_\varepsilon (x)$ converges to the point $(0, \{N-\sigma(|x|)^2\}^{1/2} ) $ as $\varepsilon \searrow 0$.
If the volume $ |U_\varepsilon (x) |_n$ of $ U_\varepsilon (x) $ with respect to the $n$-dimensional Lebesgue measure satisfies     
\begin{equation}\label{vol_lim}
\lim_{\varepsilon \downarrow 0}\frac{|U_\varepsilon (x) |_n}{\varepsilon^n V_n} =\frac{\sigma(|x|)\sigma'(|x|)}{\{ N-\sigma(|x|)^2\}^{1/2}}\left\{\frac{\sigma(|x|)}{|x|} \right\}^{n-1}, 
\end{equation}
then we find that
\begin{align*}
&f^\rho_{n,N}(x) 
=\frac{ A_{N-n}}{N^{(N-1)/2} A_N}  \lim_{\varepsilon \downarrow 0} \frac{1}{\varepsilon^n V_n}    \int_{ U_\varepsilon (x)} \left(  \frac{N}{N-|y|^2-s^2} \right)^{1/2} s^{N-n-1} dy ds\\ 
=&\frac{ A_{N-n}}{N^{(N-1)/2} A_N} \cdot \frac{\sigma(|x|)\sigma'(|x|)}{\{ N-\sigma(|x|)^2\}^{1/2}}  \left\{\frac{\sigma(|x|)}{|x|} \right\}^{n-1}  \cdot
           \left\{\frac{N}{\sigma(|x|)^2}\right\}^{1/2}   \left\{ N-\sigma(|x|)^2\right\}^{(N-n-1)/2}  \\ 
=&\frac{ A_{N-n}}{N^{n/2} A_N}  
  \left\{ 1-\frac{\sigma(|x|)^2}{N}\right\}^{(N-n-2)/2} \left\{ \frac{\sigma(|x|)}{|x|} \right\}^{n-1} \sigma'(|x|)
\end{align*}
as desired.
To prove~\eqref{vol_lim}, we need  several claims.

\begin{claim}\label{abs}
The functions $r^\pm_\varepsilon (y)$  depend only on $|y|$ not on $y$ itself.
\end{claim}
\begin{proof}
For any $y \in p_n(\Sigma (B_\varepsilon(x))) \setminus\{0\} $,
it turns out that 
\begin{align*}\label{alt}
&B^y_\varepsilon(x) 
\!=\!\left\{
x(a,b):=\left(|x| + a,  \frac{b}{|y|} y \right) \biggm|\! (a, b) \in D_\varepsilon, \  y=p_n(\Sigma(x(a,b)))   
=\frac{\sigma(|x(a,b)|)}{|x(a,b)|}  \frac{b}{|y|} y
\right\},\\ \notag
&D_\varepsilon:=\{(a,b) \in \R^2\ |\ a^2+b^2 < \varepsilon^2,\ b > 0\} .
\end{align*}
This means that $r^+_\varepsilon(y)$ (resp.\ $r^-_\varepsilon(y)$) is equal to the supremum (resp.\ infimum) of 
\[
|x(a,b)|=\left\{ (|x|+a)^2 +b^2\right\}^{1/2} 
\] 
on $D_\varepsilon$ subject to 
\begin{equation*}
|y|=Y(a,b):=
\frac{\sigma(|x(a,b)|)}{|x(a,b)|}b =|p_n(\Sigma(x(a,b)))|,
\end{equation*}
which concludes the proof of the claim. 
$\hfill \diamondsuit$
\end{proof}
We sometimes denote $r^\pm_\varepsilon(y)$ by $r^\pm_\varepsilon(|y|)$ as functions on $[0,\eta_\varepsilon)$, 
where $\eta_\varepsilon=\eta_\varepsilon(x)$ given by 
\begin{gather*}
\eta_\varepsilon:
=\sup\{|y|\ | \ y \in p_n(\Sigma (B_\varepsilon(x) ))  \}
=\sup\left\{Y(a,b) \ | \ (a,b)  \in D_\varepsilon   \right\}.
\end{gather*}
Note that  $\lim_{ \varepsilon \downarrow 0 }\eta_\varepsilon=0$. 
It follows from Claim~\ref{abs} that 
\begin{align}\label{first} 
\lim_{\varepsilon \downarrow 0} \frac{|U_\varepsilon (x)|_n}{\varepsilon^nV_n}   \notag
=&
\lim_{\varepsilon \downarrow 0} \frac{1}{\varepsilon^nV_n} \int_{p_n(\Sigma( B_\varepsilon(x))) } \int_{\{s \geq 0\ |\ N-\sigma(r^+_\varepsilon(y))^2 \leq s^2 \leq N-\sigma(r^-_\varepsilon(y))^2 \}} ds dy \\ \notag
=&\lim_{\varepsilon \downarrow 0} \frac{1}{\varepsilon^nV_n}
\int_{ p_n(\Sigma(  B_\varepsilon(x) ))} \left[ \left\{N-\sigma(r^-_\varepsilon(|y|))^2\right\}^{1/2}-  \left\{N-\sigma(r^+_\varepsilon(|y|))^2\right\}^{1/2} \right]  dy \\ \notag
=&\lim_{\varepsilon \downarrow 0} \frac{A_{n-1}}{\varepsilon^nV_n}
 \int_0^{\eta_\varepsilon}  \left[ \left\{N-\sigma(r^-_\varepsilon(\eta))^2\right\}^{1/2}-  \left\{N-\sigma(r^+_\varepsilon(\eta))^2\right\}^{1/2}  \right]  \eta^{n-2}d\eta \\  
 =&\lim_{\varepsilon \downarrow 0} \frac{A_{n-1} }{V_n} \left( \frac{\eta_\varepsilon}{\varepsilon} \right)^{n-1}  
 \int_0^1  u_\varepsilon(t) dt, 
\end{align}
where, in the last equality, we substitute $\eta=t\eta_\varepsilon $  and set 
\[
u_\varepsilon(t):= \frac{t^{n-2}}\varepsilon\left[ \left\{N-\sigma(r^-_\varepsilon(t\eta_\varepsilon))^2\right\}^{1/2}-  \left\{N-\sigma(r^+_\varepsilon(t\eta_\varepsilon))^2\right\}^{1/2}  \right].
\]
We will  investigate the limits of $\eta_\varepsilon/\varepsilon$ and $u_\varepsilon(t)$ as $\varepsilon \searrow 0$.

It is easy to check that $|x(a,b)|$ does not have extrema on  $D_\varepsilon$ by using Lagrange multipliers, and $y \neq 0$ leads to $b\neq 0$.
In other words, for any $\eta \in (0, \eta_\varepsilon)$,  there exist $a^\pm_\varepsilon(\eta) \in I_\varepsilon:=[-\varepsilon,\varepsilon]$ such that 
\begin{gather} \label{mini}
r^\pm_\varepsilon(\eta)
=\left| x \left( a^\pm_\varepsilon(\eta), \left\{ \varepsilon^2-a^\pm_\varepsilon(\eta)^2 \right\}^{1/2}   \right)\right|, \quad
\eta= Y\left( a^\pm_\varepsilon(\eta), \left\{ \varepsilon^2-a^\pm_\varepsilon(\eta)^2 \right\}^{1/2}   \right) .
\end{gather}
The monotonicity of $|x(a, ( \varepsilon^2-a^2)^{1/2})| $ in $a$ and the definition of $r^\pm_\varepsilon $ imply    
\begin{gather} \label{defa+}
a^+_\varepsilon(\eta)=\max\left\{ a \in I_\varepsilon\ |\ \eta=Y \left(a, (\varepsilon^2-a^2)^{1/2}\right)\right\}, \\ \label{defa-}
a^-_\varepsilon(\eta)=\min\left\{ a \in I_\varepsilon\ |\ \eta=Y\left(a, (\varepsilon^2-a^2)^{1/2}\right)\right\}. 
\end{gather}
Due to the fact $r^\pm_\varepsilon(0)=|x|\pm \varepsilon$,    
$a^\pm_\varepsilon(\eta)$ are extended to $[0,\eta_\varepsilon)$ by putting $a^\pm_\varepsilon(0)=\pm\varepsilon$.
\begin{claim}\label{mono}
The functions $a^\pm_\varepsilon (\eta)$ are monotone and $|a^\pm_\varepsilon (\eta)| \leq \varepsilon$ on $\eta \in (0, \eta_\varepsilon)$.
\end{claim}
\begin{proof}
From the monotonicity  
\[
\frac{\partial}{ \partial b} Y(a,b)
=\frac{\sigma'(|x(a,b)|)}{|x(a,b)|^2} b^2 +\frac{\sigma(|x(a,b)|)}{|x(a,b)|^3} (|x|+a)^2 \geq 0, 
\]
where we use the nonnegativity of $\sigma'$, we deduce  
\begin{equation}\label{max}
\eta_\varepsilon
=\sup\left\{ Y(a,b) \ |\ (a,b) \in D_\varepsilon
\right\}
=\max\left\{ Y\left( a, \left( \varepsilon^2-a^2 \right)^{1/2} \right) \Bigm| a \in I_\varepsilon \right\}.
\end{equation}
Since the function $Y( a, ( \varepsilon^2-a^2 )^{1/2} ) $ is continuous on $a \in I_\varepsilon$ and takes the value $0$ at the boundary, 
the intermediate value theorem yields that, for any $\eta_1,\eta_2 \in(0,\eta_\varepsilon)$ with $ \eta_1 < \eta_2$,  there exist $a^\pm(\eta_i)  \in I_\varepsilon$ for $i=1,2$ such that 
\[
\eta_i=Y\left(a^\pm(\eta_i) , \left\{\varepsilon^2-a^\pm(\eta_i) ^2 \right\}^{1/2}\right), \quad
a^-(\eta_1) <a^-(\eta_2) < a^+(\eta_2)< a^+(\eta_1).
\]
This with~\eqref{defa+}, \eqref{defa-}  leads to   
\[
-\varepsilon=a^-_\varepsilon(0) <a^-_\varepsilon(\eta_1) < a^-_\varepsilon(\eta_2)< a^+_\varepsilon(\eta_2) <a^+_\varepsilon(\eta_1) <a^+_\varepsilon(0)=\varepsilon.
\]
$\hfill\diamondsuit$
\end{proof}
Since Claim~\ref{mono} implies 
\[
r^\pm_\varepsilon(\eta)
=\left| x \left( a^\pm_\varepsilon(\eta), \left\{ \varepsilon^2-a^\pm_\varepsilon(\eta)^2 \right\}^{1/2}   \right)\right| 
=\left\{ |x|^2+\varepsilon^2+2a^\pm_\varepsilon(\eta)|x|\right\}^{1/2} \in [ |x|-\varepsilon, |x|+\varepsilon]
\]
for any $\eta \in (0, \eta_\varepsilon)$, the mean value theorem yields    
\begin{align*}
0 \leq u_\varepsilon(t)
   &\leq \frac{t^{n-2}}\varepsilon
 \left[  \max_{r \in [ |x|-\varepsilon, |x|+\varepsilon]} \frac{-\sigma'(r)\sigma(r)}{ \{ N-\sigma(r)^2\}^{1/2}}  
 \left\{ r^-_\varepsilon(t\eta_\varepsilon)-r^+_\varepsilon(t\eta_\varepsilon)\right\} \right]  \\ 
 &\leq t^{n-2}
 \left[ \max_{r \in [ |x|-\varepsilon, |x|+\varepsilon]} \frac{2 \sigma'(r)\sigma(r)}{ \{ N-\sigma(r)^2\}^{1/2}} \right],
\end{align*}
which ensures that $u_\varepsilon(t)$ is dominated by an integrable function on $t \in (0,1)$.
\begin{claim}\label{limit}
For any $t \in (0,1)$, the limits 
\[
\eta_\sigma=\eta_\sigma(x):=\lim_{\varepsilon \downarrow 0}\frac{\eta_\varepsilon}{\varepsilon}, \quad
a^\pm_\sigma (t):=\lim_{\varepsilon \downarrow 0} \frac{a^\pm_\varepsilon(t\eta_\varepsilon)}{\varepsilon}
\]
exist and satisfy $\eta_\sigma=\sigma(|x|)/|x|,\ a^\pm_\sigma(t)=\pm( 1-t^2)^{1/2}$.
\end{claim}
\begin{proof}
Claim~\ref{mono} leads to 
$ {a}^\pm_2(t):=\varliminf_{\varepsilon \downarrow 0} (a^\pm_\varepsilon(t\eta_\varepsilon)/\varepsilon)^2  \in[0,1] $ 
and combining \eqref{mini} with the fact $|x(a, ( \varepsilon^2-a^2 )^{1/2})| \to  |x| $ for any $a \in I_\varepsilon$ as $\varepsilon\searrow 0$ yields   
\begin{align*}
 \varlimsup_{\varepsilon \downarrow 0} \frac{t \eta_\varepsilon}{\varepsilon}
=\varlimsup_{\varepsilon \downarrow 0} \frac{1}{\varepsilon}Y\left( a^\pm_\varepsilon(t\eta_\varepsilon), \left\{ \varepsilon^2-a^\pm_\varepsilon(t\eta_\varepsilon)^2 \right\}^{1/2} \right) 
=\frac{ \sigma( |x|)}{|x|} \left\{1-\underline{a}^\pm_2(t) \right\}^{1/2} \leq \frac{ \sigma( |x|)}{|x|} 
\end{align*}
for any $t \in (0,1)$.
Letting $t \nearrow 1$, we have $\varlimsup_{\varepsilon \downarrow 0} \eta_\varepsilon/\varepsilon \leq \sigma(|x|)/|x|$.
On the other hand, it holds by~\eqref{max} that 
\[
 \varliminf_{\varepsilon \downarrow 0} \frac{\eta_\varepsilon}{\varepsilon}
\geq  \varliminf_{\varepsilon \downarrow 0} \frac{Y( 0,  \varepsilon) }{\varepsilon}  
=\varliminf_{\varepsilon \downarrow 0} \frac{\sigma ( |x(0,  \varepsilon)|) }{|x(0,  \varepsilon)|}\frac{\varepsilon}{\varepsilon}  
=\frac{\sigma(|x|)}{|x|} ,
\]
meaning $\eta_\sigma =\sigma(|x|)/|x|$ and hence $a^\pm_\sigma(t)^2={a}^\pm_2(t)= 1-t^2$.
If there exists $t_0 \in (0,1)$ such that $a^+_\sigma(t_0) =a^-_\sigma(t_0)$,  then Claim~\ref{mono} with the squeeze lemma implies  
$ a^\pm_\sigma(t) \equiv a^\pm_\sigma (t_0)$ for any $t\in(t_0,1)$, which contradicts $a^\pm_\sigma(t)^2=1-t^2$. 
We thus have $a^\pm_\sigma(t)=\pm( 1-t^2)^{1/2}$. $\hfill \diamondsuit$
\end{proof}
Since $u_\varepsilon(t)$ is dominated by an integrable function on $t \in(0,1)$ and Claim~\ref{limit} implies  
\begin{align*}
 & \lim_{\varepsilon \downarrow 0} u_\varepsilon(t) 
 =\lim_{\varepsilon \downarrow 0} \frac{t^{n-2}}{\varepsilon} \left[\left\{N-\sigma(r^-_\varepsilon(t\eta_\varepsilon))^2\right\}^{1/2}-\left\{N-\sigma(t^+_\varepsilon(s\eta_\varepsilon))^2\right\}^{1/2} \right]   \\ 
&=\frac{\sigma'(|x|)\sigma(|x|) }{\{ N-\sigma(|x|)^2\}^{1/2}}  \lim_{\varepsilon \downarrow 0} \frac{t^{n-2}}{\varepsilon }\left\{r^+_\varepsilon(t\eta_\varepsilon)-r^-_\varepsilon(t\eta_\varepsilon)\right\}\\
&=\frac{\sigma'(|x|)\sigma(|x|) }{\{ N-\sigma(|x|)^2\}^{1/2}}  \lim_{\varepsilon \downarrow 0} \frac{t^{n-2}}{\varepsilon}\! \left[ \left\{|x|^2+\varepsilon^2+2a^+_\varepsilon(t\eta_\varepsilon)|x| \right\}^{1/2}-
\left\{|x|^2+\varepsilon^2+2a^-_\varepsilon(t\eta_\varepsilon)|x|\right\}^{1/2}\right]  \\
&=\frac{2\sigma'(|x|)\sigma(|x|) }{ \{ N-\sigma(|x|)^2\}^{1/2}} t^{n-2}(1-t^2)^{1/2}   
\end{align*}
for any $t \in (0,1)$, Lebesgue's dominated convergence theorem yields    
\begin{align*}
 \lim_{\varepsilon \downarrow 0} \int_0^1 u_\varepsilon(t) dt
= \int_0^{1}  \frac{2\sigma'(|x|)\sigma(|x|)}{ \{N-\sigma(|x|)^2\}^{1/2}} t^{n-2}(1-t^2)^{1/2}  dt 
=  \frac{\sigma'(|x|)\sigma(|x|)}{ \{ N-\sigma(|x|)^2\}^{1/2}} \mathrm{B\!}\left(\frac32, \frac{n-1}2\right), 
\end{align*}
where $\mathrm{B}(\cdot,\cdot)$ is the beta function. 
According to the relation 
${V_n}/{A_{n-1}}=\mathrm{B\!}\left(3/2, (n-1)/2\right)$ 
and~\eqref{first}, we compute 
\begin{align*}
\lim_{\varepsilon \downarrow 0} \frac{|U_\varepsilon (x)|_n}{\varepsilon^nV_n}
&=\frac{A_{n-1}}{ V_n}  \lim_{\varepsilon \downarrow 0}  \left( \frac{\eta_\varepsilon}{\varepsilon} \right)^{n-1} \int_0^1 u_\varepsilon(t)dt  
=\frac{A_{n-1}}{ V_n} \cdot  \eta_\sigma^{n-1} \cdot \frac{\sigma'(|x|)\sigma(|x|)}{ \{ N-\sigma( |x|)^2\}^{1/2}  } \frac{V_n}{A_{n-1}} \\
&=\frac{\sigma(|x|)\sigma'(|x|)}{\{ N-\sigma(|x|^2)\}^{1/2}}\left\{\frac{\sigma(|x|)}{|x|}\right\}^{n-1}, 
\end{align*}
which is~\eqref{vol_lim}. 
This completes the proof of the lemma.
$\qedd$
\end{proof}
Let us now generalize the Poincar\'e limit.
\begin{proof}(Theorem~\ref{poin})
Given any $x \notin s_n^\rho (\R^n)$,  
we find that $x \notin P_{n,N}^\rho(S_N)$ for any $N \in \N$ hence $f_{n,N}^\rho(x)=0=f_n^\rho(x)$.
For any $x\in s_n^\rho(\R^n) \setminus \{0\}$, 
Lemma~\ref{express} with the relation $\lim_{N\to\infty} {A_{N-n}}{N^{n/2}}/ A_N=(2\pi)^{-n/2}$ yields 
$f^\rho_{n,N}(x) \to f^\rho_n(x)$ as $N \to \infty$.
We thus have the pointwise convergence of $f_{n,N}^\rho$ to $f_n^\rho$ on $\R^n\setminus\{0\}$ as $N \to \infty$.

It is easy to check that $f^\rho_n$ has unit mass on $\R^n$ with respect to the Lebesgue measure, 
additionally, for any $R \in \R$ satisfying $\sigma(R)^2 = 2(n+2)$ and any $N\geq 2(n+2)$, we find  
\[
 \left[ 1-\frac{\sigma(|x|)^2}{N}    \right]_+^{(N-n-2)/2}  \leq \mathbf{1}_{B_R(0)} (x)+\exp\left(-\frac{\sigma(|x|)^2}2 \right), 
\]
where $[t]_+:=\max\{t, 0 \}$ for $t \in \R$ and  $\mathbf{1}_{B_R(0)}$ is the  characteristic function  on $B_R(0)$.
This  ensures that $f^\rho_{n,N}$ is dominated by an integrable function and 
hence Lebesgue's dominated convergence theorem yields the weak convergence of $P_{n,N}^\rho{}_\sharp v_N$ to $\nu_n^\rho$ as $N \to \infty$.
$\qedd$
\end{proof}
We give a necessary and sufficient condition for $\mu_n^f$ to be a generalized Poincar\'e limit in terms of $f$.
In what follows, we denote by $\supp(f)$ the support of $f$ and set
\[
r_f:=\inf\{ r\ |\ r \in \supp(f)\}, \quad
R_f:=\sup\{ r\ |\ r \in \supp(f)\}.
\]
\begin{proposition}\label{rad}
A radial  probability measure $\mu_n^f$ is a generalized Poincar\'e limit if and only if 
$\supp(f)$ is connected, on the interior of which $f$ is continuous. 
In the case of $\mu_n^f=\nu_n^\rho$, $\supp(f)$ coincides with the closure of $s_1^\rho ((0,\infty))$. 
\end{proposition}
\begin{proof}
The ``only if" part and the claim on the support follow immediately from Theorem~\ref{poin}.
To prove the ``if" part, 
let $\sigma$ be the function on $(r_f,R_f)$ solving the equation 
\begin{equation}\label{def}
\int_{0}^{\sigma(r)} (2\pi)^{-n/2} \exp\left(-\frac{s^2}2\right) s^{n-1} ds
=\frac1{M_n^f} \int_{r_f}^r f(s)s^{n-1} ds. 
\end{equation}
Then $\sigma$ is $C^1$, strictly increasing  and $\sigma((r_f,R_f)) =(0,\infty)$, 
which ensures the existence of the function $\rho$ satisfying~\scal 
such that $s_1^\rho \circ \sigma (r)=r$ holds for any $ r \in{(r_f,R_f)}$.
For $f_n^\rho=d\nu_n^\rho/dx$, Theorem~\ref{poin} and differentiating~\eqref{def} yield that 
\[
f_n^\rho(x)
=(2\pi)^{-n/2} \exp\left( -\frac{\sigma(|x|)^2}2\right) \left\{ \frac{\sigma(|x|)}{|x|} \right\}^{n-1} \sigma'(|x|)
=\frac{f(|x|)}{M_n^f} 
\]
for almost every $x \in \supp(\mu_n^f)=\supp(\nu_n^\rho)$, that is, $\mu_n^f=\nu_n^\rho$. 
$\qedd$
\end{proof}

\section{Isoperimetric profile of Gaussian type}
We derive the isoperimetric profile of Gaussian type for the generalized Poincar\'e limit with $\rho$ satisfying~$\lip{n}$ 
by using L\'evy's isoperimetric inequality for Euclidean spheres. 
In analogy with $\R^n$, we denote by $X^\varepsilon$ the $\varepsilon$-neighborhood of $X\subset S_N$ with respect to the spherical distance function $d_{S_N}$.
\begin{proposition}{\rm (L\'evy's isoperimetric inequality~\cite{BM})}\label{SBM} 
For any $X \subset S_N$, take a closed metric ball $B \subset S_N$ with $v_N[B]=v_N[X]$.
Then we have $v_N[X^\varepsilon] \geq v_N[B^\varepsilon]$ for any $\varepsilon>0$. 
\end{proposition}

\begin{proof}(Theorem~\ref{isop})
Let $\sigma$ be the inverse function of $s_1^{\rho_1}$ and 
$F(r):=\mu_1^f[(-\infty, r]]$ be the {\it cumulative distribution function} of $\mu_1^f$, 
which is differentiable on $(r_f,R_f)$.
Since $\rho_1$ satisfies $\lip{1}$, 
we find $r_f=0$ and $\inf_{r \in {(0, R_f)}} \sigma'(r)>0$ (see~Lemma~\ref{lip1} below).
Moreover, by~\eqref{def}, it holds for any $r \in (0, R_f)$ that 
$G(\sigma(r))=F(r)$ and 
\begin{equation}\label{dist-rel}
G'(\sigma(r))=\frac{1}{\sigma'(r)}F'(r).
\end{equation}

The claim is trivial for the case of $a=0$ and $1$ since the right-hand side is equal to $0$.
We first consider the case of $a \in (1/2,1)$.
For any $A \subset \R^n$ with $\mu_n^f[A] \in(1/2,1)$, there exists a unique $\alpha \in (0,R_f)$ such that 
$\mu_n^f[A]=\mu_1^f[(-\infty,\alpha]]=F(\alpha)$.
Fix $\beta \in (0, \alpha)$.
\begin{claim}\label{nbd}
For any $t \in (0, (R_f-\alpha)/2)$, there exists $N_0 \in \N$, independent of $\beta$, such that 
\[
\left( (P_{1,N}^{\rho_1})^{-1} ( (-\infty, \beta]) \right)^{L(\beta, t)}
\supset (P_{1,N}^{\rho_1})^{-1}(( -\infty, \beta + t ] ), 
\quad
L(\beta, t):=\sup_{r \in [\beta, \beta+t]} \sigma'(r) t+ t^2
\]
holds for any $N \geq N_0$.
\end{claim}
\begin{proof}
We first remark that $L(\beta, t)$ is positive finite by $[\beta, \beta +t] \subset (0, R_f)$.
Without loss of generality, we may assume that $N \in \N$ satisfies $\alpha+t <s_1^{\rho_1}(N^{1/2})$, 
which ensures the unique existence of $x_0 \in (0, N^{1/2})$ such that $s_1^{\rho_1}(x_0)=\beta$.
Moreover, for any point $(x,\xi) \in (P_{1,N}^{\rho_1})^{-1}((-\infty, \beta+t])$ with 
$P_{1,N}^{\rho}=s_1^{\rho_1}(x) \in(\beta, \beta+t]$, we find $\xi \neq 0$.
This implies that, for $\xi_0:=\left(N-x_0^2\right)^{1/2}\xi/|\xi|$,  
it holds that $(x_0, \xi_0) \in (P_{1,N}^{\rho_1})^{-1}((-\infty, \beta])$ and 
\[
|(x,\xi)-(x_0,\xi_0)|^2=2N\left\{\! 1-\cos \left( \frac{d_{S_N}((x,\xi), (x_0,\xi_0))}{N^{1/2}} \right)\!\right\}
=d_{S_N}((x,\xi), (x_0,\xi_0))^2+O(N^{-1}).
\]
Since the direct computation provides   
\begin{align*}
|(x,\xi)-(x_0,\xi_0)|^2
&=2N\left\{ 1-\frac{xx_0}{N}-\left(1-\frac{x^2}{N}\right)^{1/2}\left(1-\frac{x_0^2}{N}\right)^{1/2} \right\} 
=\frac{|x-x_0|^2}{2N}+O(N^{-1}) 
\end{align*}
for any $t>0$, there exists $N_0\in \N$ such that
if $ N \geq N_0$, then we have 
\begin{align*}
d_{S_N}((x,\xi), (x_0,\xi_0))
&\leq |x-x_0|+ t^2
=|\sigma(s_1^\rho(x))-\sigma(s_1^\rho(x_0))| + t^2 \\
&\leq \left( \sup_{r \in [s_1^\rho(x_0), s_1^\rho(x)] }\sigma'(r) \right)|s_1^\rho(x) -s_1^\rho(x_0)| + t^2 
\leq L(\beta, t),
\end{align*}
which concludes the proof of the claim. 
$\hfill \diamondsuit$
\end{proof}
For $N \in\N$ large enough, 
we deduce from Theorem~\ref{poin} that 
\begin{equation*}
  v_N[(P_{n,N}^{\rho_n})^{-1}(A)]=(P_{n,N}^{\rho_n})_\sharp v_N [A] 
> (P_{1,N}^{\rho_1})_\sharp v_N[(-\infty, \beta]]=v_N[(P_{1,N}^{\rho_1})^{-1}((-\infty, \beta])].
\end{equation*}
The Lipschitz continuity of $P_{n,N}^{\rho_n}$ with Lipschitz constant $L_n$ deduced from $\lip{n}$ provides   
\[
v_N\left[ (P_{n,N}^{\rho_n})^{-1}\left(A^{L_n L(\beta,t) }\right) \right]
\geq
v_N\left[\left( (P_{n,N}^{\rho_n})^{-1}(A) \right)^{L(\beta,t)} \right]
\geq  
v_1 \left[ (P_{1,N}^{\rho_1})^{-1}((-\infty, \beta+t])\right],
\] 
where the last inequality follows from the fact that $(P_{1,N}^{\rho_1})^{-1}( (-\infty, \beta] )$ is a closed metric ball 
with Proposition~\ref{SBM} and Claim~\ref{nbd}.
Letting $N \to \infty$ in the above inequality together with Theorem~\ref{poin} implies  
$\nu_n^{\rho_n} \left[ A^{L_n L(\beta,t)  }\right] \geq \nu_1^{\rho_1} [(-\infty, \beta+t]]$
for any $t \in (0, (R_f-\alpha)/2)$.
Since $\beta<\alpha$ is arbitrary and $L(\beta,t)$ is continuous in $\beta$, this also holds for $\alpha$, namely 
\begin{equation}\label{bdd}
 \mu_n^f\left[ A^{L_n L(\alpha,t)  }\right]=\nu_n^{\rho_n} \left[ A^{L_n L(\alpha,t)  }\right]
\geq 
\nu_1^{\rho_1} \left[(-\infty, \alpha+t] \right]=F(\alpha+t).
\end{equation}
\begin{claim}\label{perimet}
$\mu_n^f{}^+[A] \geq I[\gamma_1](\mu_n^f[A])/L_n$.
\end{claim}
\begin{proof}
For $\varepsilon>0$ small enough, 
the continuity and monotonicity of $L(\alpha,t)$ in $t>0$ guarantees 
the existence of $t(\varepsilon)>0$ such that $L_nL(\alpha,t(\varepsilon))=\varepsilon$ 
and 
\[
\lim_{\varepsilon \downarrow 0} \frac{t(\varepsilon)}{\varepsilon}=
\lim_{t \downarrow 0} \frac{t}{L_nL(\alpha,t)}=\frac1{\sigma'(\alpha)L_n}.
\]
We then have 
\begin{align*}
\mu_n^f{}^+[ A^\varepsilon]
&=\varliminf_{\varepsilon \downarrow 0} \frac{\mu_n^f[A^\varepsilon]-\mu_n^f[A] }{\varepsilon }
\geq \varliminf_{\varepsilon \downarrow 0} \frac{ F(\alpha +t(\varepsilon))-F(\alpha)} {\varepsilon }
=\frac{F'(\alpha )}{\sigma'(\alpha)L_n} 
=\frac{I[\gamma_1](\mu_n^f[A])}{L_n}, 
\end{align*}
where the last equality follows from $G(\sigma(\alpha))=F(\alpha)=\mu_n^f[A]$, that is, $\sigma(\alpha)=G^{-1}(\mu_n^f[A])$ and~\eqref{dist-rel}.
$\hfill \diamondsuit$
\end{proof}
The arbitrary choice of $A\subset \R^n$ with $\mu_n^f[A] \in (1/2,1)$ implies $I[\mu_n^f](a) \geq I[\gamma_1](a)/L_n$ for $a\in(1/2,1)$.
Since the similar argument works for the case of $a\in(0,1/2)$ and moreover 
$a=1/2$ which is equivalent to $\alpha=0$ if $\varlimsup_{r \downarrow 0} \sigma'(r)/(2\pi)^{1/2} = \varlimsup_{r \downarrow 0} f(r)/M_1^f \in (0, \infty)$,  
we have the desired result.
$\qedd$
\end{proof}
\begin{remark}
In the case of $f(r)=\exp(-r^2/2)$, Theorem~\ref{isop} corresponds to the case of finite dimensional Gaussian measures in~\cite{Bo,ST}, 
where the case of an infinite dimensional Gaussian measure $\gamma_\infty$ was also proved.
However we may not expect to extend Theorem~\ref{isop} for infinite dimensional cases   
since $I[\gamma_\infty]$ is obtained through the fact that $\{\gamma_n\}_{n \in \N}$ is a cylinder set measure but $\{\mu_n^f\}_{n \in \N}$ is generally not a cylinder set measure.
\end{remark}

\section{Condition and Example}\label{ex}
In order to provide criteria for $\mu_n^f=\nu_n^\rho$ such that $\rho$ satisfies~$\lip{n}$ 
in terms of $f$, we first prepare the following lemma. 
\begin{lemma}~\label{lip1}
For a function $\rho$ satisfying~$\scal$ and $L>0$, the following~\eqref{L1}, \eqref{L2} and~\eqref{L3} are
equivalent to each other.
\begin{enumerate}[{\rm (i)}]
\item\label{L1} $\rho$ satisfies $\lip{1}$ and the smallest Lipschitz constant of $s_1^\rho$ is $L$.
\item\label{L2} $\rho$ satisfies $\lip{n}$ and the smallest Lipschitz constant of $s_n^\rho$ is $L$ for any $n \in \N$. 
\item\label{L3} $\lim_{r \downarrow 0} s^\rho_1(r)=s_1^\rho(0)=0$ and  $\sup_{r>0} (s_1^\rho)'(r) =  L$.
\end{enumerate}
\end{lemma}
\begin{proof}
Since~\eqref{L2}$\Rightarrow$\eqref{L1} is trivial, we show \eqref{L1}$\Rightarrow$\eqref{L3}$\Rightarrow$\eqref{L2}.
Note that,  for any $r \in \R \setminus\{0\}$, 
$(s_1^\rho)'(r)=\rho'(|r|)|r|+\rho(|r|)= (s_1^\rho)'(|r|)$
holds.

If~\eqref{L1}, then we have 
$\lim_{r \downarrow 0} s^\rho_1(r)=s_1^\rho(0)=0$ by continuity of $s_1^\rho$ and 
$\sup_{r>0} (s_1^\rho)'(r) =  L$ by the differentiability with the Lipschitz continuity of $s_1^\rho$.
We thus find~\eqref{L1}$\Rightarrow$\eqref{L3}.

Assume~\eqref{L3}.
Since the mean value theorem yields $s_1^{\rho}(r)-s_1^{\rho}(\varepsilon) \leq L(r-\varepsilon)$ for any $r \geq \varepsilon>0$,  
letting $\varepsilon \searrow 0$ and dividing by $r$ provide $\sup_{r>0} \rho(r) \leq L$. 
Let $\{ \lambda_{m}^x\}_{m=1}^n$ be the eigenvalues of the Jacobi matrix $( J_{ij}^x)_{1 \leq i,j \leq n}$ of $s_n^\rho$ at $x$.
By the differentiability of $s_n^\rho$, \eqref{L2} is equivalent to 
$\max_{1 \leq m \leq n}\sup_{  x \in \R^n\setminus\{0\}} |\lambda_{m}^x| =L$.
Fix $x=(x^i)_{i=1}^n \in \R^n\setminus\{0\}$  and let $\{v_m^x\}_{m=1}^n$ be an orthogonal basis such that $v_1=x/|x|$.
Then, for each $1 \leq m \leq  n$, $\rho'(|x|)|x|\delta_{m1}+\rho(|x|)$ is an eigenvalue $\lambda_{m}^x$ with eigenvector $v_m^x$ 
since we compute 
\[
J_{ij}^x =\rho' (|x|) \frac{x^i x^j}{|x|} + \rho(|x|) \delta_{ij},
\] 
where $\delta_{ii}=1$ and $\delta_{ij}= 0$ if $i \neq j$.
This provides     
\begin{align*}
\max_{1 \leq m \leq n}  \sup_{x \in \R^n\setminus\{0\} } \left| \lambda_m^x \right| 
=\max_{1 \leq m \leq n}  \sup_{x \in \R^n\setminus\{0\} } \left| \rho'(|x|)|x| \delta_{m1} + \rho(|x|) \right| 
=\sup_{r>0} \max\{ \rho(r),(s_1^\rho)'(r)\}= L. 
\end{align*}
$\qedd$
\end{proof}
Let us next consider the following conditions on $f$.
\begin{enumerate}[{\rm (a)}]
\item \label{condi1}
$f$ is positive on $(0,R_f)$ and $\varliminf_{r \downarrow 0}f(r)>0$.
\item \label{condi3}
There exist a continuous, positive function $\psi$ on $(R,R_f)$ for some $R>0$ such that
\begin{enumerate}[{\rm ({b}1)}]
\item\label{condic2}
$\displaystyle
\lambda f(r)\psi(r)r^{n-1} 
\leq \int_r^{R_f} f(s) s^{n-1} ds \leq \frac1\lambda  f(r) \psi(r)r^{n-1} $\quad for some $\lambda \in (0,1)$.
  \item\label{condic1}
$\displaystyle \varliminf_{r  \uparrow R_f } \{ \psi(r)^2 \ln \left(  f(r) \psi(r) r^{n-1} \right) \} >-\infty$.
\end{enumerate}
\end{enumerate}
In the case of $f(r)=\exp(-r^2/2)$, \eqref{condi3} holds for $\psi(r)=1/r$, especially, for any $\lambda \in (0,1)$
\begin{equation}\label{ineq2}
\lambda \exp\left(  -\frac{r^2}{2} \right) r^{n-2}  \leq  \int_{r}^{\infty} \exp\left( -\frac{s^2}{2} \right) s^{n-1} \leq \frac1\lambda \exp\left(  -\frac{r^2}{2} \right)  r^{n-2}
\end{equation}
holds for $r$ large enough.
\begin{proposition}\label{Lip}
For the generalized Poincar\'e limit $\mu_n^f$ with $\rho$, 
assume $\varlimsup_{r \downarrow 0}f(r) < \infty$.
Then  
$\lip{n}$ is equivalent to the combination of~\eqref{condi1} and~\eqref{condi3}.
\end{proposition}
\begin{proof}
Let $\sigma$ be the inverse function of $s_1^\rho$, then  
\begin{align}\label{rell}
&\frac{M_n^f} {(2 \pi)^{n/2}} \sigma'(r) = f(r) \exp\left( \frac{\sigma(r)^2}{2} \right) \left\{\frac{r}{\sigma(r)} \right\}^{n-1}
\end{align}
holds on $(r_f,R_f)$ by differentiating~\eqref{def}. 
We also have $\lim_{r \downarrow r_f} \sigma(r)=0, \lim_{r \uparrow R_f} \sigma(r)=\infty$ and the equivalence between 
$\inf_{r>0} (s_1^\rho)'(r)<\infty$ and $\inf_{r \in(r_f, R_f)} \sigma'(r)>0$. 

Assume~\eqref{condi1} and~\eqref{condi3}.
Since we deduce $r_f=0$ from~\eqref{condi1}, $\lim_{r \downarrow 0} s^\rho_1(r)=0$ holds.
Hence it is enough for~$\lip{n}$ to show $\varliminf_{r \downarrow 0} \sigma'(r),\ \varliminf_{r \uparrow R_f} \sigma'(r) >0$
by Lemma~\ref{lip1}(iii) together with~\eqref{rell} and~\eqref{condi1}.

The conditions~\eqref{def} and $\varlimsup_{r \downarrow 0}f(r)<\infty$ yield   
\begin{align*}
\varlimsup_{r \downarrow 0} \left[ \sigma'(r) \left\{ \frac{\sigma(r)}{r}\right\}^{n-1} \right] 
= \frac{(2 \pi)^{n/2}}{M_n^f}\varlimsup_{r \downarrow 0} \left\{ f(r) \exp\left( \frac{\sigma(r)^2}{2}\right) \right\} 
\leq \frac{(2 \pi)^{n/2}}{M_n^f} \varlimsup_{r \downarrow 0} f(r) <\infty,  
\end{align*}
which ensures the existences of $c>0$ and $r_0>0$ such that 
$\sigma'(r) \sigma(r)^{n-1} <cr^{n-1}$ on $(0,r_0)$.
Integrating it with the condition $\sigma(0)=0$ implies $ \sigma(r)^n< cr^n$ on $(0,r_0)$.
We thus have
\begin{align*}
0< \frac{(2 \pi)^{n/2}}{ M_n^f} \varliminf_{r \downarrow 0}   f(r)  
=\varliminf_{r \downarrow 0} \left[ \sigma'(r) \left\{ \frac{\sigma(r)}{r}\right\}^{n-1} \right]
\leq c^{(n-1)/n}\varliminf_{r \downarrow 0}    \sigma'(r).
\end{align*}

It follows from~\eqref{condic2}, \eqref{ineq2} and~\eqref{def} that 
\[
\lambda \exp\left(-\frac{\sigma(r)^2}{4} \right) \geq \frac{M_n^f}{(2\pi)^{n/2}}  \int_{\sigma(r)}^{\infty}  \exp\left( -\frac{s^2}{2} \right) s^{n-1} ds= \int_r^{R_f} f(s) s^{n-1} ds \geq \lambda f(r)\psi(r) r^{n-1}  
\]
for $r<R_f$ large enough, which implies 
\[
0 \leq  \varlimsup_{r  \uparrow R_f } \left\{ \psi(r)^2 \sigma(r)^2 \right\} 
\leq  -4\varliminf_{r  \uparrow R_f } \left\{  \psi(r)^2 \ln (  f(r) \psi(r)r^{n-1})\right\}  <\infty.
\]
Similarly, we deduce from~\eqref{def} with the converse inequalities of~\eqref{condic2} and~\eqref{ineq2} that 
\begin{align*}
\frac{M_n^f}{(2\pi)^{n/2}}\varliminf_{ r \uparrow R_f }\sigma'(r) 
=\varliminf_{ r \uparrow R_f } \left[ f(r)\exp\left( \frac{\sigma(r)^2}{2} \right)  \left\{ \frac{r}{\sigma(r)}\right\}^{n-1} \right]
\geq \varliminf_{ r \uparrow R_ f} \frac{\lambda^2}{\psi(r) \sigma(r) } >0.
\end{align*}
We thus have $\varliminf_{r \downarrow 0}  \sigma'(r),\ \varliminf_{r \uparrow R_f} \sigma'(r) >0$.

Conversely, suppose~$\lip{n}$.
Lemma~\ref{lip1}\eqref{L3} and Proposition~\ref{rad} yield $r_f=0$, hence $f(r)>0$ on $(0,R_f)$ holds by~\eqref{rell}.
It moreover follows from the mean value theorem that 
$\varliminf_{r \downarrow 0}  \sigma'(r) \leq \varliminf_{r \downarrow 0} \sigma(r)/r$, which means    
\begin{align*}
0<
\varliminf_{r \downarrow 0}  \sigma'(r)^n
\leq 
\varliminf_{r \downarrow 0} \left[\exp\left(-\frac{\sigma(r)^2}{2}\right) \sigma'(r) \left\{ \frac{\sigma(r)}{r}\right\}^{n-1} \right]
= \frac{(2 \pi)^{n/2}}{M_n^f} \varliminf_{r \downarrow 0} f(r).  
\end{align*}
We thus have~\eqref{condi1}.

Due to $\inf_{r\in(0,R_f)} \sigma'(r)>0$,
$\psi:=1/(\sigma\sigma' )$ is a continuous, positive function on $(0, R_f)$ 
and it holds for any $\lambda \in (0,1)$ that 
\begin{align*}
\lambda f(r)r^{n-1} \psi(r) =  
\lambda e^{-\sigma(r)^2/2}  \sigma(r)^{n-2}  
&\leq  \int_{\sigma(r)}^{\infty} e^{-s^2/2} s^{n-1}  ds 
=\frac{(2\pi)^{n/2}}{M_n^f} \int_r^{R_f} f(s) s^{n-1} ds \\
& \leq  \frac1\lambda e^{-\sigma(r)^2/2}  \sigma(r)^{n-2}  =\frac1\lambda f(r)r^{n-1} \psi(r) 
\end{align*}
for $r<R_f$ large enough, which means~\eqref{condic2}.
This with Lemma~\ref{necc} below concludes the proof of Proposition~\ref{Lip}. 
$\qedd$
\end{proof}
\begin{lemma}\label{necc}
Given a generalized Poincar\'e limit $\mu_n^f=\nu_n^\rho$, 
if $\rho$ satisfies~$\lip{n}$ and a continuous, positive function $\psi$ on $(R,R_f)$ for some $R>0$ satisfies~\eqref{condic2}, 
then~\eqref{condic1} holds.
\end{lemma}
\begin{proof}
For the inverse function $\sigma$ of $s_1^\rho$, 
we deduce from~\eqref{condic2}, \eqref{ineq2} and~\eqref{def} that 
\begin{align*}
\frac1\lambda \exp\left(- \sigma(r)^2\right) \leq \frac{M_n^f}{(2\pi)^{n/2}}  
\int_{\sigma(r)}^{\infty}  \exp\left( -\frac{s^2}{2}\right) s^{n-1} ds= \int_r^{R_f} f(s) s^{n-1} ds \leq \frac1\lambda f(r)\psi(r) r^{n-1},   
\end{align*}
hence $\psi(r)^2\sigma(r)^2 \geq -\psi(r)^2 \ln ( f(r)\psi(r)r^{n-1} )>0$.
Since~$\lip{n}$ leads to $\varliminf_{r \uparrow R_f} \sigma'(r)>0$, 
\eqref{def} with the converse inequalities of~\eqref{condic2} and~\eqref{ineq2} yields
\begin{align*}
0<\frac{M_n^f}{(2\pi)^{n/2}}\varliminf_{ r \uparrow R_f } \sigma'(r) 
=\varliminf_{ r \uparrow R_f } \left[f(r)\exp\left(\frac{\sigma(r)^2}{2}\right)  \left\{ \frac{r}{\sigma(r)}\right\}^{n-1} \right] 
\leq  \varliminf_{ r \uparrow R_ f}\frac1{\lambda^2 \psi(r) \sigma(r) } .
\end{align*}
Combining the two inequalities,  we  have~\eqref{condic1}.
$\qedd$
\end{proof} 
We give a simple criterion $f$ to satisfy~\eqref{condi3}, which concerns the logarithmic concavity (see Remark~\ref{log} below for the definition of logarithmic concavity).
\begin{lemma}\label{logcvx}
For a radial probability measure $\mu_n^f$, suppose that $f$ is $C^2$, positive on $(R,R_f)$ for some $R>0$ and $\lim_{r \uparrow R_f}f(r)=0$.
Moreover if the function $\Phi:=-\ln f$ satisfies 
$\lim_{r \uparrow R_f} \Phi''( r ) \in (0,\infty]$, $\lim_{r \uparrow R_f} \Phi'( r ) =\infty$ and 
$\varlimsup_{r \uparrow R_f} \Phi''( r )/\Phi'(r)^2<\infty$, then \eqref{condi3} holds.
\end{lemma}
\begin{proof}
Let us prove that~\eqref{condi3} holds for $\psi:=1/\Phi'$.
For any $c\in \R$, define the function by
\[
\Psi_c(r):=f(r)\psi(r)r^{n-1}-c\int_r^{R_f} f(s)s^{n-1} ds.  
\]
It turns out that $\lim_{r \uparrow R_f} \Psi_c(r)=0$ and 
\begin{align*}
\Psi_c'(r)
=f(r)r^{n-1}\left(-1 -\frac{\Phi''(r)}{\Phi'(r)^2} +\frac{n-1}{r}\psi +c \right). 
\end{align*}
This means that, for $\lambda \in(0,1) $ satisfying 
\[
\lambda < \varliminf_{r \uparrow R^f} \left(1 +\frac{\Phi''(r)}{\Phi'(r)^2} -\frac{n-1}{r}\psi \right) 
\leq  \varlimsup_{r \uparrow R^f} \left(1 +\frac{\Phi''(r)}{\Phi'(r)^2} -\frac{n-1}{r}\psi \right) <\frac1\lambda,
\]
$\Psi_{1/\lambda} (r) \leq 0 \leq \Psi_\lambda(r)$ holds for $r<R_f$ large enough, which is equivalent to~\eqref{condic2}. 
Note that the existence of $\lambda$ is guaranteed by the assumptions.

Since we compute directly 
\begin{align*}
\varliminf_{r  \uparrow R_f } \left\{ \psi(r)^2 \ln \left(  f(r) \psi(r) r^{n-1} \right) \right\}
&=\varliminf_{r  \uparrow R_f }  
  \left\{ -\frac{\Phi(r)} {\Phi'(r)^2}  - \frac{\ln \Phi'(r)}{\Phi'(r)^2}+(n-1)\frac{\ln r}{\Phi'(r)^2} \right\} \\
&\geq -\varlimsup_{r  \uparrow R_f }  \frac{\Phi(r)} {\Phi'(r)^2}  + (n-1) \varliminf_{r  \uparrow R_f }\frac{r}{\Phi'(r)^2}
\end{align*}
and find $\varliminf_{r  \uparrow R_f }\ln r/\Phi'(r)^2 \in (-\infty,\infty]$ (note that $R_f<1$ may happen), 
we only need to show $\varlimsup_{r  \uparrow R_f }\Phi(r)/\Phi'(r)^2 <\infty$ for~\eqref{condic1}.
This follows from L'H\^{o}pital's rule, that is, 
\[
\lim_{r \uparrow R_f }\frac{\Phi(r)}{\Phi'(r)^2}
=\lim_{r  \uparrow R_f }\frac{\Phi'(r)}{(\Phi'(r)^2)'}
=\lim_{r  \uparrow R_f }\frac{1}{2\Phi''(r)}<\infty.
\]
$\qedd$
\end{proof}
\begin{remark}\label{log}
An absolutely continuous probability measure $\mu$ on $\R^n$ with respect to the Lebesgue measure is said to be {\it logarithmic concave} 
if $-\log (d\mu/dx):\R^n \to (-\infty,\infty]$ is convex.
For such probability measures, Bobkov~\cite{Bob} derived the isoperimetric profile of Gaussian type.
The conditions in Lemma~\ref{logcvx} also concern the convexity of $-\log (d\mu/dx)$, 
however, we only consider the behavior of $f$ around $R_f$, not on whole $\R$.
\end{remark}

Let us show that some probability measures characterized by $\varphi$-exponential functions satisfy the conditions in Theorems~\ref{poin}, \ref{isop}.
See \cite{OT} and references therein for details of $\varphi$-exponential functions.
In what follows, $\varphi$ is always assumed to be  a continuous,  positive, non-decreasing function on $(0,\infty)$.

The {\it $\varphi$-logarithmic function} is defined for $t \in (0,\infty)$ by 
$\ln_{\varphi}(t):=\int_1^t 1/{\varphi(s)}ds$, 
which is strictly increasing. 
Set $l_{\varphi} : =\lim_{t\downarrow 0} \ln_{\varphi}(t)$ and $ L_{\varphi} := \lim_{t\uparrow \infty} \ln_{\varphi}(t)$.
The inverse function of $\ln_{\varphi}$ is called the {\it $\varphi$-exponential function} and is naturally extended to $\R$ as
\[ 
\exp_{\varphi}(\tau):=
\begin{cases}
0   &\text{if\ }  \tau  \leq l_{\varphi}, \\
\ln_{\varphi}^{-1}(\tau) &  \text{if\ }  \tau \in \left(l_{\varphi},L_{\varphi} \right), \\
\infty &  \text{if\ }  \tau \geq L_\varphi. 
\end{cases} 
\]
Define a kind of differentiable coefficients of $\varphi$ by 
\begin{align*}
\theta_{\varphi}:=\sup_{s>0} \left\{\frac{s}{\varphi(s)} \cdot \varlimsup_{\varepsilon \downarrow 0} \frac{\varphi(s+\varepsilon)-\varphi(s)}{\varepsilon}\right\}, \quad
\delta_{\varphi}:=\inf_{s>0}\left\{\frac{s}{\varphi(s)} \cdot \varlimsup_{\varepsilon \downarrow 0} \frac{\varphi(s+\varepsilon)-\varphi(s)}{\varepsilon}\right\}.
\end{align*}
\begin{remark}
The case of $\varphi(s)=s$ is the crucial case,  where the $\varphi$-exponential function coincides with the usual exponential function and $\theta_\varphi=\delta_\varphi=1$. 
Another significant case is that $\varphi(s)=s^{q}$ for $q >0$ and $q \neq 1$.
In this case,  $\theta_\varphi=\delta_\varphi=q$ and  the $\varphi$-exponential function is  the power function given by  
\[
 \exp_q (\tau):= 
 [ 1+(1-q) \tau]_+^{1/{1-q}},
\]
where $0^{a}:=\infty$ for $a<0$. 
Since we have $\exp_q (\tau) \to \exp(\tau)$ as $q \to 1$,  we regard $\exp_1(\tau):=\exp(\tau)$ for the convenience.
\end{remark}
We recall the following lemma.
\begin{lemma} {\rm (\cite[Lemmas~2.10,\ 2.12,\ Propositions~2.13,\ 2.14]{OT})} \label{OT}
Suppose $\theta_\varphi<\infty$. \\
$(1)$ 
$s^{\theta_{\varphi}}/\varphi(s)$ $($resp.\ $s^{\delta_{\varphi}}/\varphi(s))$  is non-decreasing $($resp.\ non-increasing$)$ in $s \in (0,\infty)$.\\
$(2)$ 
We have $\exp_\varphi(r) \leq \exp_{\theta_\varphi}(\varphi(1)r)$ for any $r \in \R$.\\
$(3)$
$\delta_\varphi \geq 1 \Rightarrow l_\varphi=-\infty \Rightarrow \theta_\varphi \geq 1$ 
$($or equivalently $\theta_{\varphi}<1 \Rightarrow l_{\varphi}>-\infty \Rightarrow  \delta_{\varphi}<1)$.
\end{lemma}
For $p>0$, let us consider the function of the form  $\phi_p(r):=\exp_\varphi ( -r^p/p)$ and set
\[
R_\phi:=\sup\{ r\ |\ r \in \supp( \phi_p) \}=(-p l_\varphi)^{1/p} \in (0,\infty].
\]
\begin{proposition} \label{prob}
$(1)$ $\phi_p$ satisfies $\lim_{r \downarrow 0} \phi_p(r)=1$ and~\eqref{condi1}, hence $\supp(\phi_p)$ is connected.
If moreover we assume either $l_\varphi>-\infty$ or $\theta_\varphi<(n+p)/n$ for $n \in \N$, then 
\[
\int_0^{-l_\varphi} \phi_p(r) r^{n-1} dr < \infty
\]
holds.
In this case, there exists a function $\rho$ satisfying~$\scal$ such that $\mu_n^{\phi_p}=\nu_n^{\rho}$.\\
$(2)$ Suppose that $\varphi$ is $C^1$ around $0$.\\  
$({\rm i})$ If $\theta_\varphi<1$, then $\phi_p$ satisfies~\eqref{condi3}, implying~$\lip{n}$.\\
$({\rm ii})$ If $1 \leq \delta_\varphi $,  $ 1< \theta_\varphi <(n+p)/n$ and $(\theta_\varphi-\delta_\varphi)(\theta_\varphi-1) \leq 1/p$, then~$\lip{n}$ does not hold.
\end{proposition}
\begin{proof}
(1) 
Obviously the first claim holds  and the above integral is finite if $l_\varphi$ is finite.
In the case of $l_\varphi=-\infty$, Lemma~\ref{OT}(3)  implies $\theta_\varphi  \geq 1$ and we compute for $\theta_\varphi=1$ that 
\begin{align*}
\int_0^\infty \phi_p(r) r^{n-1} dr &\leq  \int_0^\infty \exp \left( -\varphi(1)\frac{r^p}{p} \right) r^{n-1} dr  < \infty
\end{align*}
and for $\theta_\varphi>1$ that  
\begin{align*}
\int_0^\infty \phi_p(r) r^{n-1} dr &\leq  \int_0^\infty \left\{ 1-\varphi(1)(1-\theta_\varphi )\frac{r^p}{p}  \right\}^{1/(1-\theta_\varphi )} r^{n-1} dr \\
&=\frac{1}{p}\left\{ \frac{p}{\varphi(1)(\theta_\varphi -1)}\right\}^{n/p}
   \mathrm{B\!}\left( \frac{n}{p}, \frac{1}{\theta_\varphi -1}-\frac{n}{p} \right) <\infty.
\end{align*}
$(2)$ We first remark that $\phi_p$ is $C^2$ on $(R,R_\phi)$ for some $R>0$.\\
(i)
Let us show that $\Phi:=-\ln \phi_p$ satisfies the conditions in Lemma~\ref{logcvx}.
We mention that Lemma~\ref{OT} with $\theta_\varphi<1$ yields 
$0 \leq \lim_{s \downarrow 0} s/{\varphi(s)} \leq \lim_{s \downarrow 0} {s^{1-\theta_\varphi}}/{\varphi(1) } = 0$, $R_\phi<\infty$
and by definition $\delta_\varphi \leq  \varphi'(s)s/\varphi(s)  \leq \theta_\varphi$ 
if $\varphi$ is differentiable at $s$.
We directly compute that 
\begin{align*}
\Phi'(r)
=\frac{\varphi(\phi_p(r))}{\phi_p(r)}r^{p-1},\quad
\Phi''(r)
=\Phi'(r)^2
\left(\frac{p-1}{r \Phi'(r)}+ 1-\frac{\varphi'(\phi_p(r)) \phi_p(r)}{\varphi(\phi_p(r))}\right),
\end{align*}
which implies   
\begin{align*}
&\lim_{r \uparrow  R_\phi}\Phi'(r) = \lim_{s \downarrow  0} \frac{\varphi(s)}{s} {R_\phi}^{p-1}=\infty,\\
&\lim_{r \uparrow  R_\phi}\Phi''(r)
\geq \lim_{r \uparrow  R_\phi}\Phi'(r)^2\left(\frac{p-1}{r \Phi'(r)}+ 1-\theta_\varphi\right)=\infty,\\
&\varlimsup_{r \uparrow  R_\phi}\frac{\Phi''(r)}{(\Phi'(r))^2} 
\leq \lim_{r \uparrow  R_\phi} \left( \frac{p-1}{r \Phi'(r)} + (1-\delta_\varphi) \right) =1-\delta_\varphi
\end{align*}
as desired.\\
\noindent
(ii)
Let $\sigma$ be the inverse function of $s_1^\rho$ and consider the function 
\[
\Psi(r):=2\phi_p(r)r^{n-1} - \int_r^{R_\phi} \phi_p(s)s^{n-1} ds.
\]
Note that Lemma~\ref{OT} yields $l_\varphi=-\infty$, that is, $R_\phi=\infty$, and 
\begin{align*}
\lim_{r \uparrow \infty}\Psi(r)
&\leq \lim_{r \uparrow \infty}2\left\{1-\varphi(1)(1-\theta_\varphi)\frac{r^p}{p}\right\}^{1/(1-\theta_\varphi)}r^{n-1}=0,
\end{align*}
where we use the  assumption $p/(\theta_\varphi-1) > n$.
If $\Psi$ is nonpositive around $\infty$, then we have 
\[
 \frac{2 M_n^f }{(2\pi)^{n/2}} e^{-\sigma(r)^2/2}\sigma(r)^{n-2} 
\geq
\frac{M_n^f }{(2\pi)^{n/2}}\int_{\sigma(r)}^{\infty} e^{-s^2/2}s^{n-1} ds
=\int_r^{\infty} \phi_p(s)s^{n-1} ds
\geq 2\phi_p(r)r^{n-1}
\]
around $\infty$, which provides 
\begin{align*}
\varliminf_{ r \uparrow \infty} \sigma'(r) 
&=\frac{(2\pi)^{n/2}}{M_n^f }\varliminf_{ r \uparrow \infty } 
\left[ \phi_p(r)\exp\left( \frac{\sigma(r)^2}{2}\right)   \left\{ \frac{r}{\sigma(r)}\right\}^{n-1} \right]
\leq \varliminf_{ r \uparrow \infty} \frac{1}{\sigma(r)}=0.
\end{align*}
Thus~$\lip{n}$ does not hold by Lemma~\ref{lip1}.

The rest is to prove the nonpositivity of $\Psi$ around $\infty$, which is equivalent to the nonnegativity of $\Psi'$ around $\infty$.
By Lemma~\ref{OT}, there exists $c>0$ depending on only $p$ and $\varphi$ such that 
\begin{align*}
\frac{\varphi(\phi_p(r))}{\phi_p(r)}
\leq \varphi(1) \phi_p(r)^{\delta_\varphi-1} 
\leq \varphi(1)\left\{1-\varphi(1)(1-\theta_\varphi)\frac{r^p}{p}\right\}^{(\delta_\varphi-1)/(1-\theta_\varphi)} 
\leq c r^{p(\delta_\varphi-1)/(1-\theta_\varphi)} 
\end{align*}
holds around $\infty$,   
which with the assumption $(\theta_\varphi-\delta_\varphi)/(\theta_\varphi-1) \leq 1/p$ implies 
\begin{align*}
\Psi'(r)
&=
2\phi_p(r)r^{n-1}
\left(-\frac{\varphi(\phi_p(r))}{\phi_p(r)}r^{p-1}+\frac{n-1}{r}+\frac12\right)  \\
&\geq 
2\phi_p(r)r^{n-1}
\left(-cr^{p(\theta_\varphi-\delta_\varphi)/(\theta_\varphi-1)-1} +\frac{n-1}{r}+\frac12\right)  
\geq 0.
\end{align*}$\qedd$
\end{proof}

\begin{remark}
Let $\rho$ be the function such that $\mu_n^{\phi_p}=\nu_n^\rho$.
In the case of $\theta_\varphi =\delta_\varphi >1$, 
Proposition~\ref{prob}(2-ii) implies that $\lip{n}$ does not holds for any $p>0$.
However in the case of $\theta_\varphi=\delta_\varphi=1$, that is $\phi_p(r)=\exp(-r^p/p)$,
$\lip{n}$ can or cannot hold depending on $p$.
Indeed, for $\psi(r):=r^{1-p}$,~\eqref{condic2} holds for any $\lambda \in(0,1)$ and $p>0$, but ~\eqref{condic1} holds only for $p \geq 2$ 
according to $\lim_{r  \uparrow \infty}  \{\psi(r)^2 \ln \left(  f(r) \psi(r) r^{n-1} \right)\}=-\lim_{r  \uparrow \infty}  r^{2-p}/p$.
Then by Lemma~\ref{necc}, $\lip{n}$ holds if and only if $p \geq 2$.\\
\end{remark}

\end{document}